\documentclass[12pt]{amsart}
\usepackage{amsmath,amsthm,amsfonts,graphicx,listings,centernot,fullpage}

\title{Elliptic Reciprocity} 
\author{Liljana Babinkostova$^{1\S}$, Kevin Bombardier$^{2}$, Matthew Cole$^{3}$, Thomas Morrell$^{4}$, and Cory Scott$^{5}$}
\thanks{Supported by National Science Foundation grant DMS 1062857}
\thanks{$^{\S}$ Corresponding Author: liljanababinkostova@boisestate.edu}
\subjclass[2010]{14H52, 14K22, 11G07, 11G20, 11B99} 
\keywords{Elliptic curve, elliptic pair, elliptic list, elliptic cycle, elliptic reciprocity} 

\usepackage{graphicx}

\begin{document}

\begin{abstract}
The paper introduces the notions of an elliptic pair, an elliptic cycle and an elliptic list over a square free positive integer $d$. These concepts are related to the notions of amicable pairs of primes and aliquot cycles that were introduced by Silverman and Stange in \cite{SS1}. Settling a matter left open by Silverman and Stange it is shown that for $d=3$ there are elliptic cycles of length 6. For $d\neq 3$ the question of the existence of proper elliptic lists of length $n$ over $d$ is reduced to the the theory of prime producing quadratic polynomials. For $d=163$ a proper elliptic list of length 40 is exhibited. It is shown that for each $d$ there is an upper bound on the length of a proper elliptic list over $d$. The final section of the paper contains heuristic arguments supporting conjectured asymptotics for the number of elliptic pairs below integer $X$. Finally, for $d\equiv_8 3$ the existence of infinitely many anomalous prime numbers is derived from Bounyakowski's Conjecture for quadratic polynomials.
\end{abstract}

\maketitle

\section{Introduction}\label{Intro}

Throughout this paper $d$ will denote a positive square-free integer, and prime numbers $p$ and $q$ will be larger than $3$. Consider an elliptic curve group $E$ over the rational numbers. Suppose that for some fundamental discriminant $d$ the group $E$ has complex multiplication in $\mathbb{Q}(\sqrt{-d})$: If $E$ has good reduction over a finite field $K$ then the order of the corresponding elliptic curve group is computed from $\vert K\vert$ and $d$. Depending on the value of $d$ there are two, four or six possible values for this order. Some of these values may be prime. We consider the problem of finding a prime number $p$ and a fundamental discriminant $d$ such that a corresponding elliptic curve group over the field ${\mathbb F}_p$ has prime order $q$. 

In Section \ref{EPair} we examine for a given value of $d$ ordered pairs $(p,q)$ of prime numbers related as above, and introduce the notion of an \emph{elliptic pair over d}. Although the definition of the binary relation of being an elliptic pair over $d$ does not appear to introduce a symmetric relation, we prove that this relation is in fact symmetric (Theorem \ref{TheLawOfEllipticReciprocity}), a fact we refer to as ``the law of elliptic reciprocity". The relation of being an elliptic pair imposes constraints on the prime numbers so related. Our first result in this direction is given in Theorem \ref{abRelationDnot3} for the case when $d$ is not $3$, returning to the case $d=3$ in a later section. We also show that a special integer $A$ identified by Silverman and Stange for the $d=3$ case has an analogue for other values of $d$ (Theorem \ref{SilStaCor24} and Corollary \ref{AabCorollary}). This integer is useful for ruling out pairs of prime numbers that are not elliptic pairs over a given $d$, and it is also useful in identifying the $d$ over which a given pair of prime numbers could be an elliptic pair. 

The case when $d$ is $3$ in several instances has to be handled separately. Though Section \ref{EListECycle} contains results for general $d$, it is dominated by results about $d=3$. In this section we take up the question of how long directed paths (elliptic lists) or cycles (elliptic cycles) can be found in the ``elliptic pair over $d$" relation. First we establish the special place of $d=3$ in this question. If for a $d$ there is a proper elliptic cycle of length $n>3$ over $d$, then $d=3$ and $n=6$ (Theorem \ref{OnlyProper6Cycle}). And then we proceed, using several additional constraints on elliptic cycles for $d=3$ derived in this section, to search of elliptic cycles over $d=3$ not ruled out by our results. One of the main computational results of this section is finding a proper elliptic 6-cycle over $d=3$. Corollaries (\ref{4ListTwice}) and (\ref{7List}) and Theorem (\ref{6CycleMod7}) provide information that makes the search for elliptic $6$-cycles over $3$ more efficient than a brute force search. Our findings in this section clarify points left open in Remark 6.5 on \cite{SS1}: In particular, there are proper elliptic cycles of length larger than $3$ when $d=3$. 

Having determined in Section \ref{EListECycle} that for $d=3$ proper elliptic lists and proper elliptic cycles over $3$ of length 6 exist, and that this is the optimal length, and having determined that for $d\neq 3$ proper elliptic cycles cannot have length larger than $2$ we turn to the question of the length of proper elliptic lists for $d\neq 3$ in Section \ref{SecEListDNot3}. In Theorem \ref{PrimePolysDnot3} we connect the problem of existence of proper elliptic lists of length $n$ over $d\neq 3$ with quadratic polynomials generating $n$ consecutive prime values. Using a classical prime generating quadratic polynomial of Euler we exhibit a proper elliptic list of length 40 over $d=163$. In Theorem \ref{MinusDQNR} and Corollary \ref{NoEListEnot19Mod24} we derive upper bounds on the length a proper elliptic list over $d\neq 3$ can have. We do not have any indications on whether this upper bound is optimal. In anticipation of further research on this topic, we introduce three functions defined on the set of squarefree positive integers: $\mathcal{M}$, the ``maximal allowable list length" function, $\mathcal{L}$, the ``maximal achievable list length" function, and $f$, the function that measures the discrepancy between $\mathcal{M}$ and $\mathcal{L}$. On very meager evidence we suspect that this discrepancy function may be computable in terms of the class function $h$ - i.e., the values $f(d)$ may be related to the values $h(-d)$.

In the final section, Section \ref{ConjRelations}, we give heuristic arguments to support a conjecture on an asymptotic estimate in terms of $d$ and $X$ for the number of elliptic pairs over $d$ occurring below an integer $X$. Finally, in Theorem \ref{Bouniakowsky}, we derive for $d\equiv_8 3$ the existence of infinitely many anomalous primes from Bouniakowsky's Conjecture for quadratic polynomials.
  
The Appendix contains a proof of Theorem \ref{OrderJ1}, and two tables of data supporting some of our work.

\section{Elliptic Pairs}\label{EPair}

Because fields of characteristic $2$ and $3$ often must be considered separately when dealing with elliptic curves, we restrict our paper to prime numbers larger than $3$. From now on, all prime numbers $p$ and $q$ are without further comment assumed to be larger than $3$. 

\newtheorem{DefEPair}{Definition}[section]
\begin{DefEPair}\label{DefEPair} 
For a square-free positive integer $d$, an ordered pair $(p,\,q)$ of prime numbers is said to be an \textbf{elliptic pair over $d$} if there exists an elliptic curve $E$ with complex multiplication (CM) in $\mathbb{Q}(\sqrt{-d})$ having order $q$ when examined over $\mathbb{F}_{p}$. The notation $(p,\, q)_{d}$ denotes the fact that the ordered pair of prime numbers $(p,\, q)$ is an elliptic pair over $d$. The witnessing curve $E$ is called a \textbf{representative curve} of the pair. 
\end{DefEPair}

For example, $(7,13)_{3}$ is an elliptic pair because the representative curve $E : y^2 = x^3 + 3$ has order $13$ when considered over the field ${\mathbb F}_7$. 

It should be noted that Definition \ref{DefEPair} as given is not a symmetric relation in $p$ and $q$. Theorem \ref{TheLawOfEllipticReciprocity} below addresses this point.

\newtheorem{DefEPrime}[DefEPair]{Definition}
\begin{DefEPrime}\label{DefEPrime}
An \textbf{elliptic prime (over $d$)} is a prime number which is the first entry in an elliptic pair over $d$.
\end{DefEPrime}
For example, in Definition \ref{DefEPair}, $p$ is an elliptic prime over $d$. In the example following Definition \ref{DefEPair}, $7$ is an elliptic prime over $3$.

\newtheorem{OrderJ0}[DefEPair]{Lemma}
\begin{OrderJ0}\label{OrderJ0}\label{drestrict} If the prime number $p$ is an elliptic prime over $d$, and if $d\, \equiv_{4} 3$, then $d \equiv_{8} 3$.
\end{OrderJ0} 
\begin{proof}
Let $d \equiv_{4} 3$, be a given fundamental discriminant, and let $p$ be an elliptic prime number over $d$. We may assume $d>3$. Assuming that $p$ is not prime in $\mathbb{Z}(\sqrt{-d})$, let $a$ and $b$ be the unique positive integers such that $4p = a^2 + db^2$. 

Let $E$ be a representative curve witnessing that $p$ is an elliptic prime over $d$. If $E$ has CM in $\mathbb{Q}(\sqrt{-d})$, then $\#E(\mathbb{F}_{p}) = p+1 \pm a$. Then, as $p+1$ is even and $\#E(\mathbb{F}_p)$ is a prime number, $a$ must be odd. If it were the case that $d \equiv_{8} 7$, then $4 \equiv_{8} 4p = a^2 + db^2 \equiv_{8} a^2 - b^2$. But since $a$ is odd, $a^2 \equiv_8 1$, implying $b^2 \equiv_8 5$. But the congruence $x^2 \equiv_8 5$ does not have a solution, a contradiction. Thus, as $d\, \equiv_4 3$ we must have $d\,\equiv_8 3$.
\end{proof}

Among the fundamental discriminants $d$ with $d\equiv_{4} 3$, the case of $d = 3$ has to be treated separately. The curves $E$ with CM in $\mathbb{Q}(\sqrt{-3})$ are precisely those with $j(E) \in \{ 0, 54000, -12288000 \}$ \cite{AECII}. Curves with $j(E) = 54000$ have trivial 2-torsion, but some curves with $j(E) = 0, -12288000$ have trivial torsion (over $\mathbb{Q}$), so we must consider those. Curves with $j = -12288000$ are covered fully in \cite{SS1}, and they can be treated in exactly the same way as other curves with CM, so we omit a full discussion of them here. When $j(E) = 0$, $E$ has the form $E : y^2 = x^3 + k$ for $k \in \mathbb{Z}$, $\gcd(k,p)=1$. Now, $p = a^2 + 3b^2$ and $\#E(\mathbb{F}_{p})$ takes on one of six values:

\newtheorem{OrderJ1}[DefEPair]{Theorem}
\begin{OrderJ1}\label{OrderJ1}
Let $p>3$ be an odd prime and let $k \not\equiv_{p} 0$. Let $N_{p}$ denote $\#E(\mathbb{F}_{p})$, where $E$ is the elliptic curve $y^2 = x^3 + k$. Here, QR and CR denote quadratic residue and cubic residue, respectively.
\begin{enumerate}
\item[(a)] If $p \equiv_{3} 2$, then $N_{p} = p+1$.
\item[(b)] If $p \equiv_{3} 1$, write $p = a^2 + 3b^2$,
 where $a,b$ are integers with $b$ positive and $a \equiv_{3} -1$. Then \[ N_{p} = \begin{cases} p+1+2a & \text{if } k \text{ is a sixth power mod } p \\ p+1-2a & \text{if } k \text{ is a CR, but not a QR, mod } p \\ p+1-a \pm 3b & \text{if } k \text{ is a QR, but not a CR, mod } p \\ p+1+a \pm 3b & \text{if } k \text{ is neither a QR nor a CR mod } p. \end{cases} \]
\end{enumerate}
\end{OrderJ1}
\begin{proof}
This result is well-known (most previous authors formulate this theorem using $p = a^2 - ab + b^2$ instead); the proof is entirely computational and is contained in Appendix \ref{OrderJ0Appendix}. The end result is that $N_{p} = p+1 - \pi - \bar{\pi}$, where $\pi = -\chi_{6}(k)^{-1} J(\chi_{2},\chi_{3})$. This Jacobi sum is evaluated as $J(\chi_{2},\chi_{3}) = a + ib\sqrt{3}$, with $a \equiv_{3} -1$ \cite{BEW1}.
\end{proof}

It immediately follows that if $(p,q)_{3}$ is an elliptic pair, then $p \equiv_{3} 1$. Note also that if $k$ is a QR modulo $p$ then $3|N_{p}$, and if $k$ is a CR modulo $p$ then $2|N_{p}$. From this observation, two corollaries follow immediately:

\newtheorem{PrimeOrder}[DefEPair]{Corollary}
\begin{PrimeOrder}\label{PrimeOrder}
Let $p \equiv_{3} 1$ be a prime and let $E : y^2 = x^3 + k$ be an elliptic curve with $k \not\equiv_{p} 0$. If $\#E(\mathbb{F}_{p})$ is prime, then $k$ is neither a QR nor CR, except in the case where $p = 7$ and $k \equiv_7 4$ (in which case $N_{7} = 3$).
\end{PrimeOrder}

\newtheorem{NumberPrimeOrder}[DefEPair]{Corollary}
\begin{NumberPrimeOrder}\label{NumberPrimeOrder}
Let $p > 7$ be a prime, and let $E : y^2 = x^3 + k$ be an elliptic curve. Then $\#E(\mathbb{F}_{p})$ can take on at most $2$ prime values, both of which are equivalent to $1 \mod{3}$.
\end{NumberPrimeOrder}

\newtheorem{TheLawOfEllipticReciprocity}[DefEPair]{Theorem}
\begin{TheLawOfEllipticReciprocity}[The Law of Elliptic Reciprocity]\label{TheLawOfEllipticReciprocity}
If $(p,q)_{d}$ is an elliptic pair, then so too is $(q,p)_{d}$.
\end{TheLawOfEllipticReciprocity}
\begin{proof}
This is a direct consequence of Theorem 6.1 in \cite{SS1} for $d \neq 3$ and for curves of non-zero $j$-invariant for $d = 3$.

When $d = 3$ and $j = 0$, we use Theorem 7.1(d) in \cite{SS1}, which states that if $E : y^2 = x^3 + k$ has prime order $q$ modulo $p$, then $E$ has order $p$ modulo $q$ if and only if \[ \left(\frac{4k}{p}\right)_{6}\left(\frac{4k}{q}\right)_{6} = 1, \] where $\left(\frac{\cdot}{\cdot}\right)_{6}$ is the sextic residue character. Since $1/6^\text{th}$ of all $4k$ coprime to $p_{0} \neq 2,3$ have $\left(\frac{4k}{p_{0}}\right)_{6} = \frac{1 \pm i\sqrt{3}}{2}$ for a particular choice of sign, we can choose $k$ such that $\left(\frac{4k}{p}\right)_{6} = \frac{1 + i\sqrt{3}}{2}$ and $\left(\frac{4k}{q}\right)_{6} = \frac{1 - i\sqrt{3}}{2}$, so $\left(\frac{4k}{p}\right)_{6}\left(\frac{4k}{q}\right)_{6} = 1$. This demonstrates that if $(p,q)_{3}$ is an elliptic pair, then so too is $(q,p)_{3}$.
\end{proof}

One consequence of Theorem \ref{TheLawOfEllipticReciprocity} is the following:

\newtheorem{abRelationDnot3}[DefEPair]{Theorem}
\begin{abRelationDnot3}\label{abRelationDnot3}
Let $(p,q)_{d}$ be an elliptic pair with $p < q$ and $j$ of the representative curve non-zero, and $a_{p},b_{p},a_{q},b_{q}$ be the unique positive integers such that $4p = a_{p}^{2} + db_{p}^{2}$ and $4q = a_{q}^{2} + db_{q}^{2}$. Then $a_{q} = a_{p} + 2$ and $b_{q} = b_{p}$.
\end{abRelationDnot3}
\begin{proof}
We know that $q = p+1+a_{p}$ and by Theorem \ref{TheLawOfEllipticReciprocity} that $p = q+1-a_{q}$, so $a_{q} = a_{p} + 2$. Then because $\frac{1}{4}(a_{p}^{2} + db_{p}^{2}) = p$ and $\frac{1}{4}(a_{q}^{2} + db_{q}^{2}) = q$, \begin{align*} \frac{a_{p}^{2}+db_{p}^{2}}{4} + 1 + a_{p} &= \frac{a_{q}^{2}+db_{q}^{2}}{4}, \\ (a_{p}+2)^{2} + db_{p}^{2} &= a_{q}^{2} + db_{q}^{2}, \end{align*} from which is follows that $b_{q} = b_{p}$.
\end{proof}

The corresponding case of the above Theorem for $d=3$ with representative curve having $j = 0$ will be covered in Section \ref{EListECycle}. As we shall remark below, the number $A_{pq}$ treated in the following three results is very useful in analyzing the relation $(p,q)_d$.

\newtheorem{SilStaCor24}[DefEPair]{Theorem}
\begin{SilStaCor24}\label{SilStaCor24}
Let $(p,q)_{3}$ be an elliptic pair. There exists an integer $A = A_{pq}$ such that \[ A^{2} = \frac{2pq + 2p + 2q - p^{2} - q^{2} - 1}{3}. \] Choose the sign on $A$ such that $A \equiv_{4} p+q+1$. Then $E : y^2 = x^3 + g^{m}$, where $g$ is a particular primitive root modulo $q$, can have one of six orders $N_{q}$ over $\mathbb{F}_{q}$:
\begin{enumerate}
\item If $m \equiv_{6} 0$:  $N_{q} = \frac{1}{2}(p+q+1+3A)$,
\item If $m \equiv_{6} 1$:  $N_{q} = p$,
\item If $m \equiv_{6} 2$:  $N_{q} = \frac{1}{2}(p+q+1-3A)$,
\item If $m \equiv_{6} 3$:  $N_{q} = \frac{1}{2}(-p+3q+3-3A)$,
\item If $m \equiv_{6} 4$:  $N_{q} = 2q+2-p$,
\item If $m \equiv_{6} 5$:  $N_{q} = \frac{1}{2}(-p+3q+3+3A)$.
\end{enumerate}
\end{SilStaCor24}
\begin{proof}
The existence of $A$ is proven in \cite{SS1} (Corollary 7.6(a)). The rest follows from comparing the orders in terms of $A$ (Corollary 7.6(b) of \cite{SS1}) and in terms of the values given in our Theorem \ref{OrderJ0}. The six values from \cite{SS1} are:  \[\begin{cases}
q+1 \pm (q+1-p), \\
q+1 \pm \frac{1}{2}(q+1-p+3A), \\
q+1 \pm \frac{1}{2}(q+1-p-3A).
\end{cases}\]
Note that if our choice of primitive root $g$ does not make the order of $E : y^2 = x^3 + g$ over $\mathbb{F}_{q}$ equal to $p$, then we can take $g \leftarrow g^{-1}$ to produce the desired result.
\end{proof}

\newtheorem{AabCorollary}[DefEPair]{Corollary}
\begin{AabCorollary}\label{AabCorollary}
Let $(p,q)_3$ be an elliptic pair with representative of invariant $j = 0$. If we write $p = a^2 + 3b^2$, with $a \equiv_{3} -1$, and we choose the sign on $b$ such that $q = p+1+a-3b$, then $A_{pq} = a+b$.
\end{AabCorollary}
\begin{proof}
We calculate:
\begin{align*} A_{pq}^{2} &= \frac{2pq + 2p + 2q - p^{2} - q^{2} - 1}{3} & \\
&= \frac{2(a^2 + 3b^2)(a^2 + 3b^2 + 1 + a - 3b) + 2(a^2 + 3b^2) + 2(a^2 + 3b^2 + 1 + a - 3b)}{3} \\
   & \hspace{0.5in} - \frac{(a^2 + 3b^2)^{2} + (a^2 + 3b^2 + 1 + a - 3b)^{2} + 1}{3} \\
&= \frac{1}{3} \left[ (2 a^4 + 6 a^2 b^2 + 2 a^2 + 2 a^3 - 6 a^2 b + 6 a^2 b^2 + 18 b^4 + 6 b^2 + 6 a b^2 - 18 b^3) \right] \\
   & \hspace{0.5in} + \frac{1}{3} \left[ (2 a^2 + 6 b^2) + (2 a^2 + 6 b^2 + 2 + 2 a - 6 b) - (a^4 + 6 a^2 b^2 + 9 b^4) \right] \\
   & \hspace{0.5in} - \frac{1}{3} \left[ a^4 + 6 a^2 b^2 + 2 a^2 + 2 a^3 - 6 a^2 b + 9 b^4 + 6 b^2 + 6 a b^2 - 18 b^3 + 1 + 2 a - 6 b \right] \\
   & \hspace{0.5in} - \frac{1}{3} \left[ a^2 - 6 a b + 9 b^2 + 1 \right] \\
&= \frac{1}{3} \left[ 3 a^2 + 3 b^2 + 6 a b \right] = (a+b)^{2}.
\end{align*}
Then, $A_{pq} \equiv_{4} p+q+1 \equiv_{4} p + (p + 1 + a - 3b) + 1 \equiv_{4} 2p+2+a-3b$. As $p$ is odd, $4|(2p+2)$, so $A \equiv_{4} a-3b \equiv_{4} a+b$. We know that $a+b$ is odd because $p$ is odd, so $-(a+b) \not\equiv_{4} a+b$. Therefore, $A = a+b$.
\end{proof}

We can also state Corollary \ref{AabCorollary} for $d \neq 3$. 
\newtheorem{AbCorollary}[DefEPair]{Corollary}
\begin{AbCorollary}\label{AbCorollary} Let $d$ be a squarefree positive integer larger than $3$ and 
Let $(p,q)_d$ be an elliptic pair over $d$ with $p\neq q$ Then there exists an integer $A_{pq}$ such that 
\begin{equation}\label{ImportantA}
    A_{pq}^{2} = \frac{2pq + 2p + 2q - p^{2} - q^{2} - 1}{d}. 
\end{equation}
Indeed, $A_{pq} = b_p (=b_q)$.
\end{AbCorollary}
\begin{proof}
We write $4p = a_p^2+db_p^2$. Using the fact that $p<q$ and $(p,q)_d$ is an elliptic pair over $d$, we have $q=p+1+a_p$. A straightforward computation shows that $A = b_{p} = b_{q}$.
\end{proof}

The number $A_{pq}$ is useful: Given $p$, $q$ and $d$ we can compute $\frac{2pq + 2p + 2q - p^{2} - q^{2} - 1}{d}$, and then we know immediately that if this number is not a perfect square, then $(p,q)_{d}$ is not an elliptic pair. Moreover, given only $p$ and $q$ within each other's Hasse intervals, we can compute the numerator of the quantity (\ref{ImportantA}) and factor out any perfect squares to compute a $d$ for which $(p,q)_d$ could be an elliptic pair. This provides a convenient proof that $p$ and $q$ sufficiently close together form an elliptic pair for exactly one square-free value of $d$.

\section{Elliptic Lists and Elliptic Cycles}\label{EListECycle}

\newtheorem{DefEList}{Definition}[section]
\begin{DefEList}\label{DefEList}
The symbol $(p_{1},p_{2},...p_{n})_{d}$ denotes an \textbf{elliptic list of length $n$ over $d$} if each of $(p_{1},p_{2})_{d}$, $(p_{2},p_{3})_{d}$, ..., $(p_{n-1},p_{n})_{d}$ is an elliptic pair. If $p_{1},p_{2},...,p_{n}$ are all distinct primes, then the elliptic list is a \textbf{proper elliptic list of length $n$ over $d$}.
\end{DefEList}

We will see in Theorem \ref{MinusDQNR} that proper elliptic lists of length $n \ge 3$ over $d = 3$ are all represented by elliptic curves of zero $j$-invariant, so we will take $j = 0$ and $d = 3$ to be equivalent cases throughout this section.

\newtheorem{5Values}[DefEList]{Theorem}
\begin{5Values}\label{5Values}
Let $(p_{1},p_{2},p_{3},p_{4},p_{5})_{3}$ be an elliptic list with $p_{1}, p_{5} \neq p_{3}$ and $p_{2} \neq p_{4}$, then \[ p_{1}-p_{2} = p_{5}-p_{4}. \]
\end{5Values}
\begin{proof}
Let $p_{3} = a^{2} + 3b^{2}$, and choose $b$ so that $A_{p_{3}p_{4}} = a+b$. Then $p_{5} = p_{3} + 3 + 3a - 3b$. Likewise, $A_{p_{2}p_{3}} = a-b$, so $p_{1} = p_{3} + 3 + 3a + 3b$. Thus, $p_{1} - p_{5} = 6b$. By Theorem \ref{OrderJ0}, $p_{2} - p_{4} = 6b$, so $p_{1} - p_{5} = p_{2} - p_{4}$. Rearrangement of terms gives the stated result.
\end{proof}

\newtheorem{DefECycle}[DefEList]{Definition}
\begin{DefECycle}\label{DefECycle}
Let $(p_{1},p_{2},...p_{n})_{d}$ form an \textbf{elliptic $n$-cycle over $d$} (or \textbf{elliptic cycle of length $n$ over $d$}) if $(p_{1},p_{2},p_{3},...,p_{n})_{d}$ forms an elliptic list and $(p_{n},p_{1})_{d}$ forms an elliptic pair. If the list is proper, then we say that we have a \textbf{proper elliptic $n$-cycle}.
\end{DefECycle}

\newtheorem{OnlyProper6Cycle}[DefEList]{Theorem}
\begin{OnlyProper6Cycle}\label{OnlyProper6Cycle}
If a proper elliptic $n$-cycle over $d$ exists for $n \ge 3$, then $d = 3$ and $n = 6$.
\end{OnlyProper6Cycle}
\begin{proof}
When $j \neq 0$ ($d \neq 3$ and some $d = 3$ curves), this result strengthens Corollary 6.2 of \cite{SS1}, but it utilizes the exact same proof. Recall that if $j \neq 0$, then $\#E(\mathbb{F}_{p}) = p+1 \pm a$. In particular, one of these values is greater than $p$, and the other is no larger than $p$. If we have a proper elliptic cycle over $d$, it has a least element. Let $p$ be this least element, and let $(p,q)_{d}$ and $(r,p)_{d}$ be the elliptic pairs it is part of. Then either $q$ or $r$ is less than $p$, a contradiction.

Now let $d = 3$ and $j = 0$. We note that by Theorem \ref{5Values}, $p_{1}-p_{2} = p_{5}-p_{4}$. Suppose that $p_{1} = p_{4}$, i.e. we have a $3$-cycle. Then \[ p_{5}-p_{4} = p_{1}-p_{2} = p_{4}-p_{5}, \] so $p_{4} = p_{5}$, and the cycle is not proper.

Suppose that $n = 4$. then $p_{1}-p_{2} = p_{5}-p_{4}$. But because we have a $4$-cycle, $p_{1} = p_{5}$, so $p_{2} = p_{4}$, and the cycle is not proper.

Now we let $n = 5$. We have that $p_{1}-p_{2} = p_{5}-p_{4}$ and $p_{2}-p_{3} = p_{6}-p_{5}$, so \[ p_{6} - p_{4} = p_{1} - p_{3}. \] Setting $p_{1} = p_{6}$ implies that $p_{3} = p_{4}$, so the cycle is not proper.

Now let $n \ge 6$. Again, $p_{1}-p_{2} = p_{5}-p_{4}$ and $p_{2}-p_{3} = p_{6}-p_{5}$, so \[ p_{6} - p_{4} = p_{1} - p_{3}. \] But now, $p_{3}-p_{4} = p_{7}-p_{6}$, so \[ p_{1}-p_{4} = p_{7}-p_{4}, \] and thus $p_{1} = p_{7}$. This means that if a proper $n$-cycle has length at least $6$, then it has length exactly equal to $6$.
\end{proof}

Given only a single case left to check to see if there are any proper elliptic cycles other than elliptic pairs, we searched for elliptic cycles over $d=3$. The following results aid in directing such a search:

\newtheorem{4List}[DefEList]{Theorem}
\begin{4List}\label{4List}
Let $(p_{1},p_{2},p_{3},p_{4})_{3}$ be a proper elliptic list. If $p_{1} = a^2 + 3b^2$, with $a \equiv_{3} -1$, then $p_{4} = (-a-2)^2 + 3b^2$.
\end{4List}
\begin{proof}
We choose $b$ such that $p_{2} = c^{2} + 3d^{2} = p+1+a-3b$. It follows from Corollary \ref{AabCorollary} that $a + b = c + d$ for suitable choice of $d$. Solving for $c$ and $d$, we get that \[ \begin{matrix} c & = & \frac{a+3b-1}{2}, \\ d & = & \frac{a-b+1}{2}. \end{matrix} \] Note that $p_{2} + 1 + c - 3d = p_{1}+1+a-3b + 1 + \frac{1}{2}(a+3b-1) - \frac{3}{2}(a-b+1) = p_{1}$, so $p_{3} = p_{2}+1+c+3d$.

Let $p_{3} = e^2 + 3f^2$. Choose $f$ such that $c-d = e+f$. Then \[ \begin{matrix} e & = & \frac{c-3d-1}{2} & = & \frac{-a+3b-3}{2}, \\ f & = & \frac{c+d+1}{2} & = & \frac{a+b+1}{2}. \end{matrix} \] Now, $p_{3}+1+e-3f = p_{2}+1+c+3d+1+\frac{c-3d-1}{2}-3\frac{c+d+1}{2} = p_{2}$, so $p_{4} = p_{3}+1+e+3f$.

Now let $p_{4} = g^2 + 3h^2$, and choose $h$ such that $e-f = g+h$. Then \[ \begin{matrix} g & = & \frac{e-3f-1}{2} & = & \frac{(-a+3b-3)-3(a+b+1)-2}{4} & = & -a-2, \\ h & = & \frac{e+f+1}{2} & = & \frac{(-a+3b-3)+(a+b+1)+2}{4} & = & b. \end{matrix} \]
\end{proof}

\newtheorem{4ListTwice}[DefEList]{Corollary}
\begin{4ListTwice}\label{4ListTwice}
Let $(p_{1},p_{2},p_{3},p_{4},p_{5},p_{6})_{3}$ be a proper elliptic list, and let $p_{1} = a^2 + 3b^2$, with $a \equiv_{3} -1$. Then:  \begin{align*}
p_{2} &= \left(\frac{a+3b-1}{2}\right)^{2} + 3 \left(\frac{a-b+1}{2}\right)^{2}, \\
p_{3} &= \left(\frac{-a+3b-3}{2}\right)^{2} + 3 \left(\frac{a+b+1}{2}\right)^{2}, \\
p_{4} &= (-a-2)^2 + 3b^2, \\
p_{5} &= \left(\frac{-a-3b-3}{2}\right)^{2} + 3 \left(\frac{a-b+1}{2}\right)^{2}, \\
p_{6} &= \left(\frac{a-3b-1}{2}\right)^{2} + 3 \left(\frac{a+b+1}{2}\right)^{2}.
\end{align*}
\end{4ListTwice}
\begin{proof}
The values for $p_{2},p_{3},p_{4}$ are given in the proof of Theorem \ref{4List}. Note that the process to obtain $p_{6}$ from $p_{1}$ is exactly the same as the process to obtain $p_{2}$ from $p_{1}$, but the sign of $b$ is flipped. Thus, the formulae for $p_{2}$ and $p_{6}$ are the same, except that the sign on $b$ is flipped. The same relationship holds for $p_{3}$ and $p_{5}$.
\end{proof}

\newtheorem{7List}[DefEList]{Corollary}
\begin{7List}\label{7List}
There are no proper elliptic lists of length 7 over $d=3$.
\end{7List}
\begin{proof}
Note that the ``a"-component of $p_4$ in the proof of  Corollary \ref{4ListTwice} is congruent to $-1$ modulo $3$, as was the ``a" component of $p_1$. If there were a proper elliptic list of length 7, then using $p_4$ as the initial point of an elliptic list of length 4 over $3$, we see from Theorem \ref{4List} that the ``a"-component of $p_7$ would have to be $-(-a-2)-2$, which is $a$, and thus $p_7 = p_1$, contradicting that the list of length $7$ is a proper list.
\end{proof}

The following result allows some efficiency in searching for proper elliptic six cycles: 
\newtheorem{6CycleMod7}[DefEList]{Theorem}
\begin{6CycleMod7}\label{6CycleMod7}
If $p = a^2 + 3b^2$ is part of a proper elliptic $6$-cycle over $d = 3$, then $a \equiv_{7} -1$ and $7|b$.
\end{6CycleMod7}
\begin{proof}
The proof is entirely computational, taking the orders found in Corollary \ref{4ListTwice}, and then computing all the primes modulo $7$. We find that at least one is divisible by $7$, a contradiction, except in the case that $a \equiv_{7} -1$ and $7|b$. Then every prime in the cycle is equivalent to $1$ modulo $7$. The cases are given in Table \ref{6CycleMod7Table} in Appendix \ref{6CycleMod7Appendix}.
\end{proof}

Due to the structure imposed by Theorem \ref{5Values}, it is clear that we can determine all the primes in a proper elliptic $6$-cycle given a single prime in the cycle. Using divisibility by $7$ to eliminate possible cases, we wrote a computer program to search the remaining values of $a$ and $b$ to find $6$-cycles. It found them almost immediately, with the smallest being \[ (275269,274723,275227,276277,276823,276319)_{3}, \] corresponding to $(a,b) = (251,266)$ (see Corollary \ref{4ListTwice} for notation). Interestingly enough, if we let $c = 275772 = 525^2 + 3 \cdot 7^2$, and compute the six orders of Theorem \ref{OrderJ0}, then we get the six primes given above. Of course, Theorem \ref{OrderJ0} does not apply here as $c$ is composite, but it just so happens that $c+1$ is prime (this is not the case in general).

The concept of a proper elliptic $6$-cycle is similar to the notion of an \textit{aliquot cycle} defined by Silverman and Stange \cite{SS1}, except that an aliquot cycle fixes a curve $E$ such that $N_{p_{1}} = p_{2},N_{p_{2}} = p_{3},..., N_{p_{n}} = p_{1}$. It is not too difficult to see that any proper elliptic cycle can be made into an aliquot cycle. Find curves $E_{i}/\mathbb{F}_{p_{i}}$, $1 \le i \le n$ such that $N_{p_{i}} = p_{i+1}$ for $1 \le i \le n-1$ and $N_{p_{n}} = p_{1}$. The coefficients of $E_{i}$ are only defined uniquely modulo $p_{i}$, so we can use the Chinese Remainder Theorem to find the unique coefficients modulo $\prod\limits_{i=1}^{n} p_{i}$ which are equivalent to the coefficients of $E_{i}$ modulo $p_{i}$ for all $1 \le i \le n$. In the case of the list above, we can rewrite it as \[ (274723,275269,276319,276823,276277,275227)_{3} \] to get the smallest (normalized) aliquot cycle corresponding to the curve $E : y^2 = x^3 + 15$.
In the course of computing the cycle above, we also found other $6$-cycles with $3$ primes represented, each twice. Here, $p_{1} = p_{2}$, $p_{4} = p_{5}$, and $p_{3} = p_{6}$, so that $p_{1}$ and $p_{4}$ are the so-called \textit{anomalous primes}. We found $3$ cycles of this form with primes less than one million. They are:  \begin{align*}
(114661,114661,115249,115837,115837,115249)_{3}, \\
(169219,169219,169933,170647,170647,169933)_{3}, \\
(283669,283669,284593,285517,285517,284593)_{3}.
\end{align*} After these cycles, the next smallest one has primes greater than ten million.

\section{Properties of Proper Elliptic Lists}\label{SecEListDNot3}

We saw in section \ref{EListECycle} that proper elliptic lists over $d = 3$ have length no longer than six.  In this section, we explore proper elliptic lists over $d \neq 3$.  Such lists are either increasing or decreasing (because $\#\tilde{E}(\mathbb{F}_{p}) = p+1 \pm a$; see the proof of Theorem \ref{OnlyProper6Cycle}), so we will assume throughout that lists are written in ascending order.

Theorem \ref{abRelationDnot3} allows us to describe any proper elliptic list over $d \neq 3$ if we have just one of its members.  The following Theorem is an extension of Theorem \ref{abRelationDnot3}, and follows immediately from it and the definition of an elliptic list:

\newtheorem{abListRelationDnot3}{Theorem}[section]
\begin{abListRelationDnot3}\label{abListRelationDnot3}
Let $(p_{1},\ldots,p_{n})_{d}$ be a proper elliptic list over $d \neq 3$, and let $a_{p_{1}},\ldots,a_{p_{n}}$ and $b_{p_{1}},\ldots,b_{p_{n}}$ be the unique positive integers such that $4p_{i} = a_{p_{i}}^{2} + db_{p_{i}}^{2}$ for each $i = 1,\ldots,n$. Then $a_{p_{i}} = a_{p_{1}} + 2i - 2$ and $b_{p_{i}} = b_{p_{1}}$ for each $i = 1,\ldots,n$.
\end{abListRelationDnot3}

\newtheorem{PrimePolysDnot3}{Theorem}[section]
\begin{PrimePolysDnot3}\label{PrimePolysDnot3}
Let $d\neq 3$ be a squarefree positive integer. Consider a prime number $p_1$ which is of the form $4p_1 = a^2 + db^2$. If $p_1$ is the initial term of a proper elliptic list of length $n$ over $d$, then the quadratic polynomial $x^2 + ax + p_1$ has $n$ consecutive prime values for $x = 0,1,\ldots,n-1$.
\end{PrimePolysDnot3}

Theorem \ref{abListRelationDnot3} implies that any proper elliptic list $(p_1,\ldots,p_n)_{d}$ over $d \neq 3$ is generated by $n$ consecutive prime values of the quadratic polynomial in $i$ \begin{align*} & \frac{1}{4}((2i + a_{1})^{2} + db_{1}^{2}) \\ = & \; i^{2} + a_1i + p_{1} \end{align*} for $i = 0,1,\ldots,n-1$.  Conversely, any such polynomial will generate an elliptic list.  Therefore, the problem of finding proper elliptic lists over $d \neq 3$ is equivalent to finding prime-generating polynomials of the form above.

The longest proper list we've found so far is:  \begin{align*} (41, 43, 47, 53, 61, 71, 83, & 97, 113, 131, 151, 173, 197, 223, 251, 281, 313, \\ & 347, 383, 421, 461, 503, 547, 593, 641, 691, 743, 797, 853, 911,\\ & 971, 1033, 1097, 1163, 1231, 1301, 1373, 1447, 1523, 1601)_{163}. \end{align*}
This list corresponds to the famous polynomial $n^{2} + n + 41$, which is prime for $n = 0,1,\ldots,39$.

It is still unclear whether proper elliptic lists of arbitrarily long length exist, although we have an upper bound for the length given $d$ in Theorem \ref{MinusDQNR}.

\newtheorem{MinusDQNR}[abListRelationDnot3]{Theorem}
\begin{MinusDQNR}\label{MinusDQNR}
Let $d \equiv_{8} 3$, $d \neq 3$ (or $d = 3$ and $j \neq 0$). Let $z$ be the smallest prime such that $\left(\frac{-d}{z}\right) \neq -1$. Then there are no proper elliptic lists over $d$ of length $n \ge z$.
\end{MinusDQNR}
\begin{proof}
Let $(p_1,\ldots,p_n)_{d}$ be an elliptic list of length $n$. Let $a_{i},b_{i}$ be the unique positive integers such that $4p_{i} = a_{i}^{2} + db_{i}^{2}$, and note that for all $1 \le i \le n$, $a_i = a_1+2i-2$, $b_{i} = b_{1} = b$ (Theorem \ref{abRelationDnot3}).

First consider the case $\left(\frac{-d}{z}\right) = 0$.  Assume that $n \ge z$.  The numbers $a_1,a_2,\ldots,a_n$ take every value modulo $z$ because $z$ is odd and the $a_{i}$'s increment by two, so one of them, say $a_j$, will be divisible by $z$.  Therefore, \[ 4p_j = a_j^2 + d b^2 \equiv_{z} 0 + 0 \cdot b \equiv_{z} 0, \] and thus $z|p_{j}$, a contradiction unless $z=p_{j}$, which we cover below.

Now consider the case $\left(\frac{-d}{z}\right) = 1$.  Assume that $n \ge z$. As before, the numbers $a_1,a_2,\ldots,a_n$ take every value modulo $z$, so one of them, say $a_j$, will satisfy $a_j \equiv_z b\sqrt{-d}$.  Therefore, \[ 4p_j = a_j^2 + d b^2 \equiv_z -d b^2 + d b^2 \equiv_z 0, \] implying that $z|p_j$, a contradiction unless $z=p_{j}$, which we cover below.

It remains only to consider the case that one of the elements in the cycle is in fact $z$.  The smallest number which can be written as $\frac{1}{4} (a^{2} + db^{2})$ for $a,b \in \mathbb{N}$ is at least $\frac{d+1}{4}$, and $z$ is never more than this value, so we only have to look at the case $p_{1} = z = \frac{d+1}{4}$. In this case, $a_{1} = b_{1} = 1$, so $p_{i} = \frac{(2i-1)^2+d}{4}$, and $p_{z} = \frac{(2z-1)^{2} + (4z-1)}{4} = z^{2}$ is composite.
\end{proof}

In particular, we see that for $d = 3, j \neq 0$, $\left(\frac{-d}{3}\right) = 0$, so we cannot construct any proper elliptic lists of length $n \ge 3$ in this case - only elliptic pairs.

\newtheorem{NoEListEnot19Mod24}[abListRelationDnot3]{Corollary}
\begin{NoEListEnot19Mod24}\label{NoEListEnot19Mod24}
Assume that $d\equiv_{8} 3$, and $d \neq 3$. If $d \not\equiv_{24} 19$, then there do not exist any proper elliptic lists of length $n \ge 3$ over $d$.
\end{NoEListEnot19Mod24}
\begin{proof}
By hypothesis, $d \equiv_{24} 3 \text{ or } 11$. In the first case, $\left(\frac{-d}{3}\right) = 0$, and in the latter case, $\left(\frac{-d}{3}\right) = 1$.
\end{proof}

\newtheorem{DefinitionM}[abListRelationDnot3]{Definition}
\begin{DefinitionM}\label{DefinitionM}
Restrict $d$ to be a positive and square-free integer. 
\[
  \mathcal{M}(d) = \left\{\begin{tabular}{ll}
                          $\min\{z>1: z \mbox{ prime and } \left(\frac{-d}{z}\right) \neq -1\}$ & if $d\equiv_8 3$\\
                          $1$                                                                   & otherwise
                          \end{tabular}
                   \right.  
\]
$\mathcal{M}(d)$ is the \textbf{maximum allowable list-length function}.
\end{DefinitionM}
For $d \not\equiv_{8} 3$, there are no lists of any length, motivating our definition of $\mathcal{M}(d) = 1$. 
For $d \equiv_{8} 3$, the value of  $\mathcal{M}(d)$ is motivated by Theorem \ref{MinusDQNR}. Note that in the case $d = 3$, there exist proper elliptic lists of length up to $6$, so $\mathcal{M}(3) = 2$ is less than the longest proper elliptic list: This happens in this case only, although this bound holds for representative curves with $j \neq 0$.

Currently we have no guarantee that any list of length $\mathcal{M}(d)-1$ exists over $d$, motivating our next definition:
\newtheorem{DefinitionL}[abListRelationDnot3]{Definition}
\begin{DefinitionL}\label{DefinitionL}
Restrict $d$ to be a positive and square-free integer. 
Define 
\[
  \mathcal{L}(d) = \max\{n | (p_{1},...,p_{n})_{d} \text{ is a proper elliptic list}\}.
\]
Also define
\[
  f(d) = \mathcal{M}(d) - \mathcal{L}(d).
\]
\end{DefinitionL}

Except for $d=3$ we have $\mathcal{L}(d)<\mathcal{M}(d)$ and $f(d)>0$. We suspect that $f(d)$ is related to the class number $h(-d)$ in some way, but we have not been able to prove any results so far, other than the fact that $f(d) = 1$ if $h(-d) = 1$.  

\section{The Relationship between Elliptic Pairs and Current Conjectures}\label{ConjRelations}

The distribution of elliptic pairs is of primary importance because it determines the time and space complexity of algorithms for generating elliptic curves of prime order (see \cite{BS1}, for example). Although thus far nobody has explicitly formulated a conjecture for this, many similar heuristics and conjectures (see \cite{BS1},\cite{SS1},\cite{K1},\cite{Z1}) suggest that the number of elliptic pairs $(p,q)_{d}$ with primes $p \le q$ less than $X$ should be asymptotic to \[ C_{d} \frac{X}{log^{2} X} \] for some constant $C_{d}$ as $X \rightarrow \infty$. It is believed that $C_{d}$ is greater than zero for all positive square-free integers $d \equiv_{8} 3$, but so far it is not even known if the number of elliptic pairs for any given $d$ is infinite.

The remark at the end of Section \ref{EPair} indicates that for every pair of primes within each other's Hasse interval, we have an elliptic pair. By the prime number theorem, for any given $p$, the probability that another number $q$ near it is prime is approximately $1/\log p$, so $p$ forms an elliptic pair with about $4\sqrt{p}/\log p$ primes, of which roughly $2\sqrt{p}/\log p$ are greater than $p$ (so that we only consider unique elliptic pairs, since Theorem \ref{TheLawOfEllipticReciprocity} indicates that the order of $p$ and $q$ does not matter). Since $d \le 4p-1$ and $d \equiv_8 3$, we have roughly $p/2$ values of $d$ which could make $(p,q)_d$ an elliptic pair, but each square-free $d$ is equivalent to $\frac{1}{2}\sqrt{\frac{p}{d}}$ of these values (obtained by multiplying $d$ by the squares of odd numbers), so we'd expect any particular $d$ to be chosen with probability $(2\sqrt{p}/\log p)/\left( \frac{p}{2} / \left(\frac{1}{2}\sqrt{\frac{p}{d}}\right) \right) = 2/\sqrt{d}\log p$. Since $\pi(X) \sim \frac{X}{\log{X}}$, this suggests to us that for any given $d$, the number of elliptic pairs with primes less than sufficiently large $X$ should be about \[ \sum\limits_{p \le X} \frac{2}{\sqrt{d}\log p} \sim \sum\limits_{n=1}^{X/\log X} \frac{2}{\sqrt{d} \log (n \log n)} \sim \sum\limits_{n=1}^{X/\log X} \frac{2}{\sqrt{d}\log n} \asymp \frac{1}{\sqrt{d}} \frac{X}{log^{2} X}. \] Because the above sum diverges as $X \rightarrow \infty$, we see that for at least one value of $d$, there must be an infinite number of elliptic pairs. Furthermore, we would expect that $C_d/C_{d^\prime} \approx \sqrt{d^\prime/d}$ as a first-order approximation. Unfortunately, this seems to contradict experimental data obtained by Silverman and Stange in Section 9 of \cite{SS1}, which shows a reciprocal relationship for $d$'s with class number $h(-d) = 1$. As $-d$ grows larger, however, the frequency of elliptic pairs decreases in general. Therefore, it is likely that our heuristic must be altered to include the class number in some way. Replacing $\frac{1}{\sqrt{d}}$ by $\frac{\sqrt{d}}{[h(-d)]^2}$ gives the closest results to our analysis, although we have no heuristic reason why this formula should be chosen.

We computed $d$ for every possible pair of primes $(p,q)$, each less than $10^7$, and examined the results for all $d \equiv_8 3$ such that $h(-d) = 1,2,3,4$. We found that $C_d \sim C \frac{\sqrt{d}}{[h(-d)]^2}$ as $C_d \rightarrow \infty$, in agreement with our guess, with $C \approx 0.16$. Unfortunately, $\frac{\sqrt{d}}{[h(-d)]^2} < 13$ for all the values we looked at, so our results are ineffective in computing the actual value of $C_d$. Furthermore, the value $C_3 \approx 1.757$ is about $6.26$ times greater than expected, so it is likely that the $d=3$ case must be examined using a different heuristic, just as it had to be treated separately earlier. Our results are given in Figure \ref{AsymptoticsGraph}, with the data listed in Appendix \ref{AsymptoticsData}.

\begin{figure}[h]
\begin{center}
\includegraphics[scale=0.35]{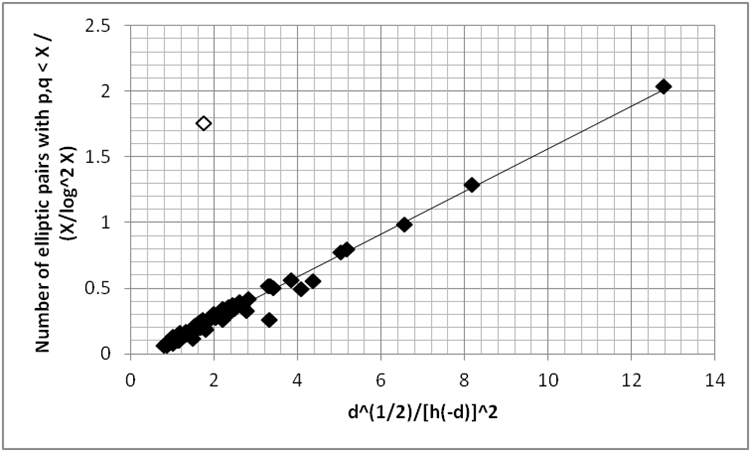}
\caption{Plot of the number of elliptic pairs $(p,q)_d$ with $p,q < X = 10^7$ divided by $\frac{X}{\log^2 X}$ as a function of $\frac{\sqrt{d}}{[h(-d)]^2}$ for $d \equiv_8 3$ and $h(-d) \le 4$. Note that we had to treat the $d=3$ case (open diamond) separate from the others, as always. The linear fit suggests that our asymptotic relation, $C_d \asymp \frac{\sqrt{d}}{[h(-d)]^2}$ is correct as $C_d \rightarrow \infty$.}
\label{AsymptoticsGraph}
\end{center}
\end{figure}

Elliptic pairs are also related to the Bouniakowsky conjecture for quadratic polynomials.
The Bouniakowsky conjecture states that all irreducible polynomials $p(x)$ with integer coefficients of degree greater than or equal to two such that there does not exist a prime $m$ which divides $p(x)$ for all integers $x$ take on an infinite number of prime values. Such polynomials are called \textit{Bouniakowsky polynomials}. The case for polynomials of degree one was proven by Dirichlet.

\newtheorem{Bouniakowsky}{Theorem}[section]
\begin{Bouniakowsky}\label{Bouniakowsky}
The Bouniakowsky conjecture (for quadratic polynomials) implies that there are an infinite number of anomalous primes (and thus elliptic pairs) over any $d \equiv_{8} 3$.
\end{Bouniakowsky}
\begin{proof}
Let $d = 3$, and let $(p,p)_{3}$ be an elliptic pair. Then $p = p+1+a-3b$, so $a=3b-1$. As $p = a^2 + 3b^2$, we have that \[ p(b) = (3b-1)^{2} + 3b^2 = 12b^2 - 6b + 1, \] which is easily checked to be a Bouniakowsky polynomial.

If $d \neq 3$, then $(p,p)_{d}$ is an elliptic pair if and only if $p = p+1-a$, or $a = 1$. Thus, $p = \frac{1}{4}(1+db^{2})$. We know that $b$ must be odd so that $p$ is an integer, so we set $b = 2c+1$ and then \[ p(c) = \frac{1}{4}(1+d(2c+1)^{2}) = dc^{2} + dc + \frac{1+d}{4}. \] Since $4|(1+d)$, this polynomial has integer coefficients, and it is easy to check that it is a Bouniakowsky polynomial.
\end{proof}

\appendix
\section{Proof of Theorem \ref{OrderJ0}}
\label{OrderJ0Appendix}

We follow the proof for curves of the form $E : y^2 = x^3 - kx$ given in \cite{W1}.

Let $E : y^2 = x^3 + k$ be an elliptic curve defined over $\mathbb{F}_p$, with $k \in \mathbb{F}_p^\times$ and $p \ge 5$ prime. If $p \equiv_3 2$, then every value modulo $p$ is a cubic residue, so $x^3$ takes on every value modulo $p$, and $x^3 + k$ does as well. Of the values modulo $p$, $y^2 = 0$ has only one solution, while $y^2 = a \in \mathbb{F}_p^\times$ has either zero or two solutions, depending on whether $a$ is a quadratic residue or non-residue. The number of each is the same, so this gives us $p$ points. Adding the point at infinity $\mathcal{O}$, we have that \[ \#E(\mathbb{F}_p) = p+1. \]

Now let $p \equiv_3 1$. Choose a primitive root $g$ modulo $p$ and let $\chi_6(g^j) = e^{ij\pi/3}$ be a Dirichlet character. Then $\chi_6^2 = \chi_3$ and $\chi_6^3 = \chi_2$.

\newtheorem{FirstLemmaOrderJ0}{Lemma}[section]
\begin{FirstLemmaOrderJ0}\label{FirstLemmaOrderJ0}
Let $p \equiv_3 1$ be prime and let $x \in \mathbb{F}_p^\times$. Then
\[
\#\{ u \in \mathbb{F}_p^\times \vert u^t = x \} = \sum\limits_{\ell=0}^{t-1} \chi_t(x)^\ell
\]
for $t \in \{ 2,3,6 \}$.
\end{FirstLemmaOrderJ0}
\begin{proof}
Since $p \equiv_3 1$, we also have that $p \equiv_6 1$, so there are six sixth roots of one in $\mathbb{F}_p^\times$. Therefore, if there is a solution to $u^6 = x$, there are six solutions. Write $x \equiv_p g^j$. Then $x$ is a sixth power modulo $p$ if and only if $6 \vert j$. We have
\[
\sum\limits_{\ell=0}^{5} \chi_6(x)^\ell = \sum\limits_{\ell=0}^{5} e^{ij\ell\pi/3},
\]
which evaluates to six if $6 \vert j$ and to zero otherwise, giving us exactly the number of $u$ for which $u^6 \equiv_p x$. This proves the lemma for $t = 6$; the proofs for $t = 2$ and $t = 3$ follow similarly.
\end{proof}

\newtheorem{SecondLemmaOrderJ0}[FirstLemmaOrderJ0]{Lemma}
\begin{SecondLemmaOrderJ0}\label{SecondLemmaOrderJ0}
Let $p \equiv_3 1$ be prime. Then
\[
\sum\limits_{b \in \mathbb{F}_p^\times} \chi_6(b)^\ell
\]
is $p-1$ if $6 \vert \ell$ and is zero otherwise.
\end{SecondLemmaOrderJ0}
\begin{proof}
If $6 \vert \ell$, then all the terms in the sum are $1$, so the sum is $p-1$. Otherwise, $\chi_6(g)^\ell \neq 1$. Multiplying by $g$ permutes the elements of $\mathbb{F}_p^\times$, so
\[
\chi_6(g)^\ell \sum\limits_{b \in \mathbb{F}_p^\times} \chi_6(b)^\ell = \sum\limits_{b \in \mathbb{F}_p^\times} \chi_6(gb)^\ell = \sum\limits_{c \in \mathbb{F}_p^\times} \chi_6(c)^\ell,
\]
which is the original sum. As $\chi_6(g)^\ell \neq 1$, the sum must be zero.
\end{proof}

We now show that the number of points on $E$ can be expressed in terms of Jacobi sums. By separating out the terms where $x = 0$ and $y = 0$, we find that the number of points is
\[
\#\{\mathcal{O}\} + \#\{ y : y^2 = k \} + \#\{ x : x^3 = -k \} + \sum\limits_{a+b=k; a,b \neq 0} \#\{ y : y^2 = a \} \#\{ x : x^3 = -b \}.
\]
By Lemma \ref{FirstLemmaOrderJ0}, the first three terms evaluate to $1$, $\sum\limits_{j=0}^{1} \chi_2(k)^j$, and $\sum\limits_{\ell = 0}^{2} \chi_3(-k)^\ell$, respectively. The latter summation expands to
\begin{align*}
&\sum\limits_{a \neq 0,k} \sum\limits_{j=0}^{1} \chi_2(a)^j \sum\limits_{\ell=0}^{2} \chi_3(a-k)^\ell \\
= &\sum\limits_{a \neq 0,k} [\chi_2(a)^0 + \chi_2(a)^1][\chi_3(a-k)^0 + \chi_3(a-k)^1 + \chi_3(a-k)^2] \\
= &\sum\limits_{a \neq 0,k} [1 + \chi_2(a) + \chi_3(a-k) + \chi_2(a)\chi_3(a-k) + \chi_3(a-k)^2 + \chi_2(a)\chi_3(a-k)^2] \\
= &(p-2) - \chi_2(k) - \chi_3(-k) + \sum\limits_{a \neq 0,k} \chi_2(a) \chi_3(a-k) - \chi_3(-k)^2 + \sum\limits_{a \neq 0,k} \chi_2(a) \chi_3(a-k)^2.
\end{align*}
We have used Lemma \ref{SecondLemmaOrderJ0} to replace the sums of single characters $\chi$ with the negative value $\chi(\pm k)$ that was omitted from the sum. Combining with the terms from before, this simplifies to
\begin{align*}
& p+1 + \sum\limits_{a \neq 0,k} \chi_2(a) \chi_3(a-k) + \sum\limits_{a \neq 0,k} \chi_2(a) \chi_3(a-k)^2 \\
= & p+1 + \chi_6(-1)^2 \sum\limits_{a \neq 0,k} \chi_2(a) \chi_3(k-a) + \chi_6(-1)^4 \sum\limits_{a \neq 0,k} \chi_2(a) \chi_3(k-a)^2 \\
= & p+1 + \chi_6(k)^{-1} \sum\limits_{a/k \neq 0,1} \chi_2(a/k)\chi_3((k-a)/k) + \chi_6(k)^{-1} \sum\limits_{a/k \neq 0,1} \chi_2(a/k)\chi_3((k-a)/k)^2 \\
= & p+1 + \chi_6(k)^{-1} J(\chi_2, \chi_3) + \chi_6(k)^{-1} J(\chi_2, \chi_3^2),
\end{align*}
where $J(\chi_2,\chi_3)$ os the Jacobi sum.

Note that if we write \[ \#E(\mathbb{F}_p) = p+1 - \alpha - \bar{\alpha}, \] then $\alpha = -\chi_6(k)^{-1}J(\chi_2,\chi_3)$, which is evaluated in \cite{BEW1}.

\section{Proof of Theorem \ref{6CycleMod7}}
\label{6CycleMod7Appendix}

Let $p_1 = a^2 + 3b^2$, with $a \equiv_3 -1$ and the sign on $b$ chosen such that $p_2 = p_1 + 1 + a - 3b$, and assume that $(p_1,p_2,p_3,p_4,p_5,p_6)$ is a proper elliptic cycle. By Corollary \ref{4ListTwice}, $a$ and $b$ determine the cycle uniquely. We enumerate over all the possible choices of $a$ and $b$ modulo $7$ in Table \ref{6CycleMod7Table}. In the last column, we take $\prod\limits_{i=1}^{6} p_i \mod 7$. If this is $0$, then at least one $p_i$ is not prime (since $7$ is not part of a proper elliptic cycle of length $6$). Only when $a \equiv_7 6 \equiv_7 -1$ and $b \equiv_7 0$ do we get a $1$ in the last column and the possibility of a non-trivial proper elliptic cycle.

\begin{tiny} \begin{table}[htbp]
  \centering
  \caption{Primes in Elliptic Cycles Modulo $7$}
    \begin{tabular}{|cc|cccccc|c|}
    \hline
    $a \mod 7$ & $b \mod 7$ & $p_1 \mod 7$ & $p_2 \mod 7$ & $p_3 \mod 7$ & $p_4 \mod 7$ & $p_5 \mod 7$ & $p_6 \mod 7$ & $\prod\limits_{i=1}^{6} p_i \mod 7$ \\
    \hline
    0     & 0     & 0     & 1     & 3     & 4     & 3     & 1     & 0 \\
    0     & 1     & 3     & 1     & 3     & 0     & 2     & 0     & 0 \\
    0     & 2     & 5     & 0     & 2     & 2     & 0     & 5     & 0 \\
    0     & 3     & 6     & 5     & 0     & 3     & 4     & 2     & 0 \\
    0     & 4     & 6     & 2     & 4     & 3     & 0     & 5     & 0 \\
    0     & 5     & 5     & 5     & 0     & 2     & 2     & 0     & 0 \\
    0     & 6     & 3     & 0     & 2     & 0     & 3     & 1     & 0 \\
    \hline
    1     & 0     & 1     & 3     & 0     & 2     & 0     & 3     & 0 \\
    1     & 1     & 4     & 3     & 0     & 5     & 6     & 2     & 0 \\
    1     & 2     & 6     & 2     & 6     & 0     & 4     & 0     & 0 \\
    1     & 3     & 0     & 0     & 4     & 1     & 1     & 4     & 0 \\
    1     & 4     & 0     & 4     & 1     & 1     & 4     & 0     & 0 \\
    1     & 5     & 6     & 0     & 4     & 0     & 6     & 2     & 0 \\
    1     & 6     & 4     & 2     & 6     & 5     & 0     & 3     & 0 \\
    \hline
    2     & 0     & 4     & 0     & 6     & 2     & 6     & 0     & 0 \\
    2     & 1     & 0     & 0     & 6     & 5     & 5     & 6     & 0 \\
    2     & 2     & 2     & 6     & 5     & 0     & 3     & 4     & 0 \\
    2     & 3     & 3     & 4     & 3     & 1     & 0     & 1     & 0 \\
    2     & 4     & 3     & 1     & 0     & 1     & 3     & 4     & 0 \\
    2     & 5     & 2     & 4     & 3     & 0     & 5     & 6     & 0 \\
    2     & 6     & 0     & 6     & 5     & 5     & 6     & 0     & 0 \\
    \hline
    3     & 0     & 2     & 6     & 0     & 4     & 0     & 6     & 0 \\
    3     & 1     & 5     & 6     & 0     & 0     & 6     & 5     & 0 \\
    3     & 2     & 0     & 5     & 6     & 2     & 4     & 3     & 0 \\
    3     & 3     & 1     & 3     & 4     & 3     & 1     & 0     & 0 \\
    3     & 4     & 1     & 0     & 1     & 3     & 4     & 3     & 0 \\
    3     & 5     & 0     & 3     & 4     & 2     & 6     & 5     & 0 \\
    3     & 6     & 5     & 5     & 6     & 0     & 0     & 6     & 0 \\
    \hline
    4     & 0     & 2     & 0     & 3     & 1     & 3     & 0     & 0 \\
    4     & 1     & 5     & 0     & 3     & 4     & 2     & 6     & 0 \\
    4     & 2     & 0     & 6     & 2     & 6     & 0     & 4     & 0 \\
    4     & 3     & 1     & 4     & 0     & 0     & 4     & 1     & 0 \\
    4     & 4     & 1     & 1     & 4     & 0     & 0     & 4     & 0 \\
    4     & 5     & 0     & 4     & 0     & 6     & 2     & 6     & 0 \\
    4     & 6     & 5     & 6     & 2     & 4     & 3     & 0     & 0 \\
    \hline
    5     & 0     & 4     & 3     & 1     & 0     & 1     & 3     & 0 \\
    5     & 1     & 0     & 3     & 1     & 3     & 0     & 2     & 0 \\
    5     & 2     & 2     & 2     & 0     & 5     & 5     & 0     & 0 \\
    5     & 3     & 3     & 0     & 5     & 6     & 2     & 4     & 0 \\
    5     & 4     & 3     & 4     & 2     & 6     & 5     & 0     & 0 \\
    5     & 5     & 2     & 0     & 5     & 5     & 0     & 2     & 0 \\
    5     & 6     & 0     & 2     & 0     & 3     & 1     & 3     & 0 \\
    \hline
    \textbf{6} & \textbf{0} & \textbf{1} & \textbf{1} & \textbf{1} & \textbf{1} & \textbf{1} & \textbf{1} & \textbf{1} \\
    6     & 1     & 4     & 1     & 1     & 4     & 0     & 0     & 0 \\
    6     & 2     & 6     & 0     & 0     & 6     & 5     & 5     & 0 \\
    6     & 3     & 0     & 5     & 5     & 0     & 2     & 2     & 0 \\
    6     & 4     & 0     & 2     & 2     & 0     & 5     & 5     & 0 \\
    6     & 5     & 6     & 5     & 5     & 6     & 0     & 0     & 0 \\
    6     & 6     & 4     & 0     & 0     & 4     & 1     & 1     & 0 \\
    \hline
    \end{tabular}%
  \label{6CycleMod7Table}%
\end{table} \end{tiny}

\section{Data for Asymptotic Evaluation of $C_d$}
\label{AsymptoticsData}

We conjecture that the number of elliptic pairs $(p,q)_d$ with $p \le q$ less than some given upper bound $X$ is asymptotic to $C_d \frac{X}{\log^2 X}$ as $X \rightarrow \infty$, with $C_d \asymp \frac{\sqrt{d}}{[h(-d)^2]^2}$ as $C_d \rightarrow \infty$. Of course, $C_d$ is bounded, and it is frequently very small, so this relation is computationally ineffective. The data, shown in Figure \ref{AsymptoticsGraph}, support this conjecture, however. The rest of this appendix consists of a table (Table \ref{AsymptoticsDataTable}) of the data we collected for values of $d$ with class number $h(-d) \le 4$.

\begin{tiny} \begin{table}[htbp]
  \centering
  \caption{$Y = $ the Number of Elliptic Pairs $(p,q)_d$ with $p \le q < X = 10^7$}
    \begin{tabular}{|cc|ccc|}
    \hline
    $d$ & $h(-d)$ & $\frac{\sqrt{d}}{[h(-d)]^2}$ & $Y$ & $\frac{Y}{X/\log^2 X}$ \\
    \hline
    3     & 1     & 1.732051 & 67619 & 1.756694 \\
    \hline
    11    & 1     & 3.316625 & 10125 & 0.263040 \\
    19    & 1     & 4.358899 & 21466 & 0.557672 \\
    43    & 1     & 6.557439 & 38158 & 0.991318 \\
    67    & 1     & 8.185353 & 49662 & 1.290184 \\
    163   & 1     & 12.76715 & 78517 & 2.039817 \\
    35    & 2     & 1.479020 & 4545  & 0.118076 \\
    51    & 2     & 1.785357 & 7054  & 0.183258 \\
    91    & 2     & 2.384848 & 12324 & 0.320169 \\
    115   & 2     & 2.680951 & 14274 & 0.370829 \\
    123   & 2     & 2.772634 & 12669 & 0.329132 \\
    187   & 2     & 3.418699 & 19490 & 0.506337 \\
    235   & 2     & 3.832427 & 21643 & 0.562270 \\
    267   & 2     & 4.085034 & 19062 & 0.495217 \\
    403   & 2     & 5.018715 & 29974 & 0.778704 \\
    427   & 2     & 5.165995 & 30647 & 0.796188 \\
    59    & 3     & 0.853461 & 2456  & 0.063805 \\
    83    & 3     & 1.012270 & 3238  & 0.084121 \\
    107   & 3     & 1.149342 & 3845  & 0.099890 \\
    139   & 3     & 1.309981 & 6503  & 0.168943 \\
    211   & 3     & 1.613982 & 8458  & 0.219733 \\
    283   & 3     & 1.869178 & 10504 & 0.272887 \\
    307   & 3     & 1.946824 & 10947 & 0.284395 \\
    331   & 3     & 2.021489 & 10676 & 0.277355 \\
    379   & 3     & 2.163102 & 11513 & 0.299100 \\
    499   & 3     & 2.482034 & 13267 & 0.344667 \\
    547   & 3     & 2.598670 & 15323 & 0.398081 \\
    643   & 3     & 2.817494 & 16295 & 0.423333 \\
    883   & 3     & 3.301702 & 20009 & 0.519820 \\
    907   & 3     & 3.346271 & 19969 & 0.518781 \\
    155   & 4     & 0.778119 & 2541  & 0.066013 \\
    195   & 4     & 0.872765 & 3510  & 0.091187 \\
    203   & 4     & 0.890488 & 3019  & 0.078432 \\
    219   & 4     & 0.924916 & 3678  & 0.095552 \\
    259   & 4     & 1.005842 & 5135  & 0.133404 \\
    291   & 4     & 1.066170 & 4255  & 0.110542 \\
    323   & 4     & 1.123263 & 3944  & 0.102462 \\
    355   & 4     & 1.177590 & 6290  & 0.163410 \\
    435   & 4     & 1.303541 & 5661  & 0.147069 \\
    483   & 4     & 1.373579 & 6076  & 0.157850 \\
    555   & 4     & 1.472402 & 6347  & 0.164891 \\
    595   & 4     & 1.524539 & 8283  & 0.215187 \\
    627   & 4     & 1.564998 & 7084  & 0.184037 \\
    667   & 4     & 1.614146 & 9325  & 0.242257 \\
    715   & 4     & 1.671218 & 9364  & 0.243270 \\
    723   & 4     & 1.680541 & 7687  & 0.199703 \\
    763   & 4     & 1.726403 & 10009 & 0.260027 \\
    795   & 4     & 1.762234 & 7752  & 0.201392 \\
    955   & 4     & 1.931442 & 10565 & 0.274471 \\
    1003  & 4     & 1.979386 & 11738 & 0.304945 \\
    1027  & 4     & 2.002928 & 11536 & 0.299697 \\
    1227  & 4     & 2.189285 & 10072 & 0.261664 \\
    1243  & 4     & 2.203513 & 13164 & 0.341992 \\
    1387  & 4     & 2.327653 & 14006 & 0.363866 \\
    1411  & 4     & 2.347705 & 12649 & 0.328612 \\
    1435  & 4     & 2.367587 & 13188 & 0.342615 \\
    1507  & 4     & 2.426256 & 14544 & 0.377843 \\
    1555  & 4     & 2.464593 & 14049 & 0.364983 \\
    \hline
    \end{tabular}%
  \label{AsymptoticsDataTable}%
\end{table} \end{tiny}

\vskip 10truept

\begin{flushleft}
 $^1$ \small {Department of Mathematics, Boise State University, Boise, ID 83725}\\
 $^2$ \small {Department of Mathematics, Statistics, and Physics, Wichita State University, Wichita, KS 67260}\\
 $^3$ \small {Department of Mathematics, University of Notre Dame, Notre Dame, IN 46556}\\
 $^4$ \small {Department of Mathematics, Washington University, St. Louis, MO 63130}\\
 $^5$ \small {Department of Mathematics and Computer Science, Colorado College, Colorado Springs, CO 80903}
\end{flushleft}

\end{document}